\documentclass[12pt]{amsart}

\usepackage{amsfonts,amsmath,amssymb,amsthm}

\usepackage{enumerate}

\usepackage{graphicx}

\makeatletter
\@namedef{subjclassname@2010}{%
  \textup{2010} Mathematics Subject Classification}
\makeatother

\newtheorem{lemma}{Lemma}
\newtheorem{corollary}{Corollary}
\newtheorem{theorem}{Theorem}

\newcommand{\Li}{\operatorname{Li}}

\newcommand{\Z}{\mathbb Z}
\newcommand{\N}{\mathbb N}
\newcommand{\Q}{\mathbb Q}
\newcommand{\R}{\mathbb{R}}
\newcommand{\C}{\mathbb{C}}
\newcommand{\F}{\mathbb{F}}
\newcommand{\av}{\underline{a}}
\newcommand{\bv}{\underline{b}}
\newcommand{\nv}{\underline{n}}
\newcommand{\xv}{\underline{x}}
\newcommand{\Tv}{\underline{T}}
\newcommand{\chiv}{\underline{\chi}}
\newcommand{\alphav}{\underline{\alpha}}
\newcommand{\zv}{\underline{0}}
\newcommand{\Mod}[1]{\ (\text{mod}\ #1)}

\begin{document}


\baselineskip=17pt



\title{Average $r$--rank Artin's Conjecture}

\author[L. Menici]{Lorenzo Menici}
\address{Dipartimento di Matematica, Universit\`a Roma Tre,\\
         Largo S. L. Murialdo, 1, I--00146 Roma Italia}
\email{menici@mat.uniroma3.it}

\author[C. Pehlivan]{Cihan Pehlivan}
\address{Dipartimento di Matematica, Universit\`a Roma Tre,\\
         Largo S. L. Murialdo, 1, I--00146 Roma Italia}
\email{cihanp@gmail.com}

\date{}

\begin{abstract}
Let $\Gamma\subset\Q^*$ be a finitely generated subgroup and let $p$ be a prime such that the reduction
group $\Gamma_p$ is a well defined subgroup of the multiplicative group $\F_p^*$. We prove an asymptotic formula for the average
of the number of primes $p\le x$ for which the index  $[\F_p^*:\Gamma_p]=m$.
The average is performed over all finitely generated subgroups $\Gamma=\langle a_1,\dots,a_r \rangle\subset\Q^*$,
with $a_i\in\Z$ and $a_i\le T_i$, with a range of uniformity $T_i>\exp(4(\log x \log\log x)^{\frac{1}{2}})$ for every $i=1,\dots,r$.
We also prove an asymptotic formula for the mean square of the error terms in the asymptotic formula with a similar range
of uniformity. The case of rank $1$ and $m=1$ corresponds to the classical Artin's conjecture for primitive roots
and has already been considered by Stephens in 1969.
\end{abstract}



\maketitle

\section{Introduction}

Artin's conjecture for primitive roots (1927) states that for any integer $a\neq0,\pm 1$ which is not a perfect square there exist
infinitely many prime numbers $p$ for which $a$ is a primitive root modulo $p$.
In particular, Artin conjectured that the number of primes not exceeding $x$ for which $a$ is a primitive root, $N_a(x)$,
asymptotically satisfies
$$
N_a(x)\sim A(a) \Li(x)\;,\qquad\text{as }x\rightarrow\infty,
$$
where $\Li(x)$ is the logarithmic integral and the positive constant $A(a)$ depends on the integer $a$.
A breakthrough in this area has been achieved by Hooley's paper \cite{H} in which Artin's conjecture has been proved under the assumption
of the Generalized Riemann Hypothesis (GRH) for the Dedekind zeta function over the Kummer extension $\Q(a^{1/k},\zeta_k)$
for any positive square-free integer $k$.
Several generalizations of the original Artin's conjecture have been studied by many authors during the following years
(for an exhaustive survey see \cite{M}).
A first unconditional result on Artin's conjecture in the 3--rank case was found by Gupta and Ram Murty \cite{GM}, 
improved few years later by Heath-Brown \cite{HB}.

In the case of rank $r=1$, a first study of the average behavior of $N_a(x)$ was proposed by Stephens \cite{St} in 1969:
he proved that, if $T>\exp(4(\log x \log\log x)^{1/2})$, then
\begin{equation}\label{ste1}
\frac1{T}\sum_{a\leq T}N_a(x) = \sum_{p\leq x} \frac{\varphi(p-1)}{p-1} + O\left(\frac{x}{(\log x)^D}\right) 
=A \Li(x)+O\left(\frac{x}{(\log x)^D}\right)\;,
\end{equation}
where $\varphi$ is the Euler totient function, $A=\prod_p \left(1-\frac1{p(p-1)}\right)$ is the Artin's constant and 
$D$ is an arbitrary constant greater than $1$.
If $T>\exp(6(\log x \log\log x)^{1/2})$, Stephens also proved that
\begin{equation}\label{ste2}
\frac1{T}\sum_{a\leq T}\left\{N_a(x)-A\Li(x)\right\}^2 \ll \frac{x^2}{(\log x)^{D'}}\;,
\end{equation}
for any constant $D'>2$.
In 1976, Stephens refined his results with different methods \cite{St2}, getting both the asymptotic bounds (\ref{ste1}) and (\ref{ste2}) under the weaker assumption $T>\exp(C(\log x )^{1/2})$, with $C$ positive constant.

If we set, for any $a\in\N\setminus\{0,\pm1\}$ and $m\in\N$, $N_{a,m}(x)$ to be the number of primes $p\equiv1\Mod m$ not exceeding $x$ 
such that 
the index $[\F_p^*:\langle a\Mod p\rangle]=m$, then for $T>\exp(4(\log x \log\log x)^{1/2})$ Moree \cite{M2} showed that
\begin{equation}\label{moree}
\frac1{T}\sum_{a\leq T}N_{a,m}(x) = \sum_{\substack{p\le x\\ p\equiv1\Mod m}}
\frac{\varphi((p-1)/m)}{p-1}+O\left(\frac{x}{(\log x)^E}\right)\;,
\end{equation}
for any constant $E>1$.


In the present work, we will discuss the average version of the \emph{$r$--rank Artin's quasi primitive root conjecture}, 
adapting the methods used by Stephens in \cite{St} to the case of rank $r$.
Let $\Gamma\subset\Q^*$ be a multiplicative subgroup of finite rank $r$. For almost all primes, namely those primes
$p$ such that for all $g\in\Gamma$ the $p$--adic valuation $v_p(g)=0$, one can consider the reduction group
$$
\Gamma_p=\{g\Mod p: g\in\Gamma\}
$$ which is a  well defined subgroup of the multiplicative group $\F_p^*$.
We denote by $N_{\Gamma,m}(x)$ the number of primes   $p\equiv1\Mod m$ not exceeding $x$ for which the index $[\F_p^*:\Gamma_p]=m$. 
It was proven by Cangelmi, Pappalardi and Susa (\cite{P}, \cite{CP} and \cite{PS1}), 
assuming the GRH for $\Q(\zeta_k,\Gamma^{1/k})$ for any natural number $k$,
that for any $\varepsilon>0$, if $m\le x^{\frac{r-1}{(r+1)(4r+2)}-\varepsilon}$, then
$$N_{\Gamma,m}(x)=\left(\delta_{\Gamma}^m+O\left(\frac{1}{\varphi(m^{r+1}\log^r x}\right)\right)\Li(x),\qquad\text{as }x\rightarrow\infty,$$
where $\delta_{\Gamma}^m$ is a rational multiple of
$$
C_r=\sum_{n\geq 1}\frac{\mu(n)}{n^r\varphi(n)}=\prod_p\left(1-\frac{1}{p^r(p-1)}\right)\;.
$$

Here we restrict ourselves to studying subgroups $\Gamma=\langle a_1,\cdots,a_r\rangle$, with $a_i \in \Z$ for all $i=1,\ldots,r$,
and we prove the following Theorems:
\begin{theorem}\label{firstorder}
Assume $T^*:=\min\{T_i:i=1,\dots,r\}>\exp(4(\log x\log\log x)^{\frac12})$ and $m\le(\log x)^D$ for an arbitrary positive constant $D$.
Then
$$
\frac{1}{T_1\cdots T_r}\sum_{\substack{a_i\in\Z\\ 0<a_1\le T_1\\ \vdots \\ 0<a_r\le T_r}} N_{\langle a_1,\cdots,a_r\rangle,m}(x)=
C_{r,m}\operatorname{Li}(x)+O\left(\frac x{(\log x)^M}\right)\;,
$$
where $C_{r,m}=\sum_{n\geq 1}\frac{\mu(n)}{(nm)^r\varphi(nm)}$ and $M>1$ is arbitrarily large.
\end{theorem}
\begin{theorem}\label{secondorder}
Let $T^*>\exp(6(\log x\log\log x)^{\frac12})$ and $m\le(\log x)^D$ for an arbitrary positive constant $D$. Then
$$
\frac{1}{T_1\cdots T_r}\sum_{\substack{a_i\in\Z\\ 0<a_1\le T_1\\ \vdots \\ 0<a_r\le T_r}}\left\{N_{\langle a_1,\cdots,a_r\rangle,m}(x)-
\,C_{r,m}\operatorname{Li}(x)\right\}^2\ll\frac{x^2}{(\log x)^{M'}}
$$
where $M'>2$ is arbitrarily large.
\end{theorem}

Notice that, since $\varphi(mn)=\varphi(m)\varphi(n)\gcd(m,n)/\varphi(\gcd(m,n))$ and $\gcd(m,n)$ is a multiplicative function of $n$
 for any fixed integer $m$, we have the following Euler product expansion:
\begin{eqnarray*}
C_{r,m}&=&\frac1{m^{r}\varphi(m)}\sum_{n\geq 1}\frac{\mu(n)}{n^{r}\varphi(n)}\prod_{p\mid \gcd(m,n)}\left(1-\frac1{p}\right)\\
&&=\frac1{m^{r+1}}\prod_{p\mid m}\left(1-\frac{p}{p^{r+1}-1}\right)^{-1} C_r\;.
\end{eqnarray*}

The results found in the present paper (see in particular equation (\ref{Smx}) and Lemma~\ref{MainTerm}) will lead as a side product to the asymptotic identity \begin{equation}\nonumber 
\frac{1}{T_1\cdots T_r}\sum_{\substack{a_i\in\Z\\ 0<a_1\le T_1\\ \vdots \\ 0<a_r\le T_r}} N_{\langle a_1,\cdots,a_r\rangle,m}(x)=
\sum_{\substack{p\le x\\ p\equiv1\Mod m}}
\frac{J_r((p-1)/m)}{(p-1)^r}+O\left(\frac{x}{(\log x)^M}\right),
\end{equation}
if $T_i>\exp(4(\log x\log\log x)^{\frac12})$ for all $i=1,\ldots,r$, $m\le(\log x)^D$ and $M>1$ arbitrary constant,
where $$J_r(n)=n^r\prod_{\substack{\ell\mid n \\ \ell\; \rm{prime}}}\left(1-\frac1{\ell^r}\right)$$ 
is the so called \emph{Jordan's totient function}. 
This provides a natural generalization of Moree's result in \cite{M2}.

Theorem~\ref{secondorder} leads to the following Corollary:
\begin{corollary}\label{exceptions}
For any $\epsilon>0$, let $$\mathcal{H}:=\{\av\in \Z^r: 0<a_i\leq T_i, i\in\{1,\dots,r\}, |N_{\av,m}(x)-C_{r,m}\Li(x)|>\epsilon\Li(x) \}\;;$$
then, supposing $T^*>\exp(6(\log x \log\log x)^{1/2})$, we have $\#\mathcal{H} \leq K|\Tv|/\epsilon^2(\log x)^F$,
for every positive constant $F$.
\end{corollary}
\begin{proof}[Proof of Corollary \ref{exceptions}]
The proof of this Corollary is a trivial generalization of that in \cite{St} (Corollary, page 187).
\end{proof}


\section{Notations and conventions}
 In order to simplify the formulas, we introduce the following notations.
Underlined letters stand for general $r$-tuples defined within some set, e.g. $\av=(a_1,\dots,a_r) \in (\F_p^*)^r$ or
$\Tv=(T_1,\dots,T_r) \in (\R^{>0})^r$; moreover, given two $r$-tuples, $\av$ and $\nv$, their scalar product is
$\av \cdot \nv = a_1 n_1 + \dots + a_r n_r$. The null vector is $\underline{0} = \{0,\dots,0\}$.
Similarly, $\chiv=(\chi_1,\dots,\chi_r)$ is a $r$-tuple of Dirichlet characters
and, given $\av \in \Z^r$, we denote the product $\chiv(\av) = \chi_1(a_1)\cdots\chi_r(a_r) \in \C$.

In addition, $(q,\av):=(q,a_1,\dots,a_r)=\gcd(q,a_1,\dots,a_r)$; otherwise, to avoid possible misinterpretations,
we will write explicitly $\gcd(n_1,\dots,n_r)$ instead of $(\nv)$. Given any $r$-tuple $\av \in \Z^r$, we indicate with $$\langle \av \rangle_p:=\langle a_1\Mod{p},\dots,a_r\Mod{p}\rangle$$
the reduction modulo $p$ of the subgroup $\langle \av \rangle = \langle a_1,\dots,a_r \rangle \subset \Q$; 
if $\Gamma= \langle a_1,\dots,a_r\rangle$, then $\Gamma_p=\langle \av \rangle_p$.

In the whole paper, $\ell$ and $p$ will always indicate prime numbers. Given a finite field $\F_p$, then $\F_p^*=\F_p\setminus\{0\}$
and $\widehat{\F_p^*}$ will denote its relative dual group (or character group).
Finally, given an integer $a$, $v_p(a)$ is its $p$-adic valuation.


\section{Lemmata}

Let $q>1$ be an integer and let $\nv\in\Z^r$. We define the \textit{multiple Ramanujan sum} as
$$
c_q(\nv):=\sum_{\substack{\av\in(\Z/q\Z)^r\\ (q,\av)=1}}e^{2\pi i \av \cdot \nv /q}\;.
$$
It is well known (see \cite[Theorem~272]{HW}) that, given any integer $n$,
\begin{equation}\label{uno}
c_q(n)= \mu\left(\frac{q}{(q, n)}\right)\frac{\varphi(q)}{\varphi\left(\frac{q}{(q, n)}\right)}.
\end{equation}
In the following Lemma, we generalize the previous result.
\begin{lemma} \label{MRS}
Let $$
J_r(m):=m^r\prod_{\ell\mid m}\left(1-\frac1{\ell^r}\right)
$$
be the Jordan's totient function, then
$$
c_q(\nv)=\mu\left(\frac{q}{(q,\nv)}\right)\frac{J_r(q)}{J_r\left(\frac{q}{(q,\nv)}\right)}\;.
$$
\end{lemma}

\begin{proof}
Let us start by considering the case when $q=\ell$ is prime. Then

\begin{eqnarray*}
c_\ell(\nv)&=&
\sum_{\av\in(\Z/\ell\Z)^r\setminus\{\zv\}}
e^{2\pi i\av \cdot \nv /\ell}\\
&&=-1+ \prod_{j=1}^r\sum_{a_j=1}^\ell
e^{2\pi i a_j n_j /\ell}=\begin{cases}
                          -1 &\text{if }\ell\nmid \gcd(n_1,\cdots,n_r)\,,\\
                          \ell^r-1&\text{otherwise.}
                         \end{cases}
\end{eqnarray*}

Next we consider the case when $q=\ell^k$ with $k\ge2$ and $\ell$ prime. We need to show that
$$
c_{\ell^k}(\nv)=\begin{cases}
                              0 &\text{if }\ell^{k-1}\nmid \gcd(n_1,\cdots,n_r)\,,\\
                              -\ell^{r(k-1)}  &\text{if }\ell^{k-1}\| \gcd(n_1,\cdots,n_r)\,,\\
                              \ell^{rk}\left(1-\frac{1}{\ell^r}\right)  &\text{if }\ell^{k}\mid \gcd(n_1,\cdots,n_r)\,.\\
                             \end{cases}
$$
To prove that, we start writing
\begin{eqnarray*}
c_{\ell^k}(\nv)&=&
\sum_{\substack{\av\in(\Z/\ell^k\Z)^r\\ (\ell,\av)=1}}
e^{2\pi i\av \cdot \nv /\ell^k}
\\
&=& c_{\ell^k}(n_1)
\prod_{j=2}^r\sum_{a_j=1}^{\ell^k}
e^{2\pi i a_jn_j/\ell^k}
+c_{\ell^k}(n_2,\ldots,n_r)
\sum_{j=1}^k
\sum_{\substack{a_1\in\Z/\ell^k\Z\\ (a_1,\ell^k)=\ell^j}}
e^{2\pi i a_1n_1/\ell^k}
\\
&=&c_{\ell^k}(n_1)
\prod_{j=2}^r\sum_{a_j=1}^{\ell^k}
e^{2\pi i a_jn_j/\ell^k}
+c_{\ell^k}(n_2,\ldots,n_r)
\sum_{j=1}^k
c_{\ell^{k-j}}(n_1)\;.
\end{eqnarray*}
If we apply (\ref{uno}), we obtain
\begin{eqnarray*}c_{\ell^k}(n_1,\ldots,n_r)&=&
 \mu\left(\frac{\ell^k}{(\ell^k, n_1)}\right)\frac{\varphi(\ell^k)}{\varphi\left(\frac{\ell^k}{(\ell^k, n_1)}\right)}
\prod_{j=2}^r\sum_{a_j=1}^{\ell^k}
e^{2\pi i a_jn_j/\ell^k}\\
&&+c_{\ell^k}(n_2,\ldots,n_r)
\sum_{j=1}^k
 \mu\left(\frac{\ell^{k-j}}{(\ell^{k-j}, n_1)}\right)\frac{\varphi(\ell^{k-j})}{\varphi\left(\frac{\ell^{k-j}}{(\ell^{k-j}, n_1)}\right)}\;.
\end{eqnarray*}

\noindent Now, for $k\ge2$, let us distinguish the two cases:
\begin{enumerate}
 \item {$\ell^{k-1}\nmid\gcd(n_1,\ldots,n_r)$}\,,
  \item {$\ell^{k-1}\mid\gcd(n_1,\ldots,n_r)$}\,.
\end{enumerate}
In the fist case we can assume, without loss of generality, that $\ell^{k-1}\nmid n_1$. Hence $\mu\left(\frac{\ell^k}{(\ell^k, n_1)}\right)=0$ and
if $k_1=v_\ell(n_1)<k-1$, then
$$
\mu\left(\frac{\ell^{k-j}}{(\ell^{k-j}, n_1)}\right)=\mu(\ell^{\max\{0,k-k_1-j\}})=\begin{cases}
                                                                                     0& \text{ if } 1\le j\le k-k_1-2,\\
                                                                                     -1& \text{ if } j=k-k_1-1,\\
                                                                                     1& \text{ if } j\ge k-k_1.
                                                                                    \end{cases}
$$
Hence
$$\sum_{j=1}^k
 \mu\left(\frac{\ell^{k-j}}{(\ell^{k-j}, n_1)}\right)\frac{\varphi(\ell^{k-j})}{\varphi\left(\frac{\ell^{k-j}}{(\ell^{k-j}, n_1)}\right)}
 =
 -\ell^{k_1}+\sum_{j=k-k_1}^{k}
 \varphi(\ell^{k-j})=0.
$$
In the second case, from the definition of $c_{q}(\nv)$ we find
$$c_{\ell^k}(\nv)=\ell^{r(k-1)}\,c_{\ell}\left(\frac{n_1}{\ell^{k-1}},\ldots,\frac{n_r}{\ell^{k-1}}\right)
=\begin{cases}
  \ell^{rk}\left(1-\frac{1}{\ell^r}\right) &\text{if }\ell^k\mid \gcd(n_1,\ldots,n_r)\,,\\
  -\ell^{r(k-1)}&\text{if }\ell^{k-1}\| \gcd(n_1,\ldots,n_r)\,.
 \end{cases}
$$
So, the formula  holds for the case $q=\ell^k$.

Finally, we claim that if $q',q''\in\N$ are such that $\gcd(q',q'')=1$, then
$$c_{q'q''}(\nv)=c_{q'}(\nv)\,c_{q''}(\nv)\,;$$
this amounts to saying that the multiple Ramanujan sum is multiplicative in $q$.
Indeed
\begin{eqnarray*}
 &\sum_{\substack{\av\in(\Z/q'\Z)^r\\ (q',\av)=1}}e^{2\pi i \av \cdot \nv/q'}&
 \sum_{\substack{\bv\in(\Z/q''\Z)^r\\ (q'',\bv)=1}}e^{2\pi i\bv\cdot \nv /q''}\\
  &&=\sum_{\substack{\av\in(\Z/q'\Z)^r\\ \bv\in(\Z/q''\Z)^r\\
\gcd(q',\av)=1\\ \gcd(q'',\bv)=1}}e^{2\pi i[n_1(q''a_1+q'b_1)+\cdots+n_r(q''a_r+q'b_r)]/(q'q'')}
 \end{eqnarray*}
and the result follows from the remark that, since $\gcd(q',q'')=1$,
\begin{itemize}
 \item for all $j=1,\ldots r$,
as $a_j$ runs through a complete set of residues modulo $q'$ and as $b_j$ runs through a complete set of residues modulo $q''$,
$q''a_j+q'b_j$ runs through a complete set of residues modulo $q'q''$.
 \item for all $\av\in(\Z/q'\Z)^r$ and for all $\bv\in(\Z/q''\Z)^r$,
\begin{eqnarray*}&\gcd(q',\av)=1&\text{ and }\gcd(q'',\bv)=1\\
&&\quad\Longleftrightarrow\quad
\gcd(q'q'',q'b_1+q''a_1',\ldots,q'b_r+q''a_r)=1.
\end{eqnarray*}
\end{itemize}
The proof of the Lemma now follows from the multiplicativity of $\mu$ and of $J_r$.
\end{proof}

From the previous Lemma we deduce the following Corollary:
\begin{corollary} \label{DirichletSums}
Let $p$ be an odd prime, let $m\in\N$ be a divisor of $p-1$. Given a $r$-tuple $\chiv=(\chi_1,\ldots,\chi_r)$  of Dirichlet
characters modulo $p$, we set
$$
c_m(\chiv):=\frac{1}{(p-1)^r}\sum_{\substack{\alphav\in(\F_p^*)^r\\ [\F_p^*:\langle\alphav\rangle_p]=m}}\chiv(\alphav)\;.
$$
Then
\begin{equation}\label{cmchiv}
\begin{split}
c_m(\chiv)=&\frac{1}{(p-1)^r}\,\mu\left(\frac{p-1}{m\gcd\left(\frac{p-1}m,\frac{p-1}{\operatorname{ord}(\chi_1)},\ldots,\frac{p-1}{\operatorname{ord}(\chi_r)}\right)}\right)\\ &\times
\frac{J_r\left(\frac{p-1}{m}\right)}{J_r\left(\frac{p-1}{m\gcd\left(\frac{p-1}m,\frac{p-1}{\operatorname{ord}(\chi_1)},\ldots,\frac{p-1}{\operatorname{ord}(\chi_r)}\right)}\right)}\;.
\end{split}
\end{equation}

\end{corollary}

\begin{proof}
Let us fix a primitive root $g\in\F_p^*$. For each $j=1,\ldots, r$, let $n_j\in\Z/(p-1)\Z$ be such that
 $$
\chi_j=\chi_j(g)=e^{\frac{2\pi i n_j}{p-1}}\;;
$$
if we write $\alpha_j=g^{a_j}$ for $j=1,\ldots,r$, then
$$
[\F_p^*:\langle\alphav\rangle_p]=m\quad\Longleftrightarrow\quad (p-1,\av)=m\;.
$$
Therefore, naming $t=\frac{p-1}{m}$, we have
\begin{equation}\label{cmchi}
\begin{split}
c_m(\chiv)&=
\frac{1}{(p-1)^r}\sum_{\substack{\av\in(\F_p^*)^r\\ \left(p-1,\av\right)=m}}\chi_1(g)^{a_1}\cdots\chi_r(g)^{a_r}=
\frac{1}{(p-1)^r}\sum_{\substack{\av'\in\left(\Z/t\Z\right)^r\\ \left(t,\av'\right)=1}} e^{2\pi i \av'\cdot \nv/t}\\ &=
\frac{1}{(p-1)^r}\,c_{\frac{p-1}m}(\nv).
\end{split}
\end{equation}
By definition we have that $\operatorname{ord}(\chi_j)=(p-1)/\gcd( n_j,p-1)$, so
$$
\frac{p-1}{m\gcd\left(\frac{p-1}m,\nv\right)}=
\frac{p-1}{m\gcd\left(\frac{p-1}m,\frac{p-1}{\operatorname{ord}(\chi_1)},\ldots,\frac{p-1}{\operatorname{ord}(\chi_r)}\right)}
$$
and this, together with Lemma 1, concludes the proof.
\end{proof}

For a fixed rank $r$, define $R_p(m):=\#\{\av\in (\Z/(p-1)\Z)^r : (\av,p-1)=m \}$.
Then using well-known properties of the M\"obius function, we can write
$$
R_p(m)=\sum_{\av \in \left(\frac{\Z}{(p-1)\Z}\right)^r} \sum_{\substack{ n\mid \frac{a_1}{m} \\ \vdots \\ n\mid \frac{a_r}{m} \\ n\mid \frac{p-1}{m} }} \mu(n)
= \sum_{n\mid \frac{p-1}{m}} \mu(n) [h_m(n)]^r \;,
$$
where
$$
h_m(n)=\# \left\{ a \in \frac{\Z}{(p-1)\Z} : n\mid \frac{a}{m}   \right\} = \frac{p-1}{nm} \;,
$$
so that
$$
R_p(m)= \left(\frac{p-1}{m}\right)^r \sum_{n\mid \frac{p-1}{m}} \frac{\mu(n)}{n^r}=J_r \left(\frac{p-1}{m}\right) \;.
$$

Defining 
\begin{equation} \label{Smx}
\begin{split}
S_m(x)&:=\frac1{m^r}\sum_{\substack{p\leq x \\ p\equiv1\Mod{m}}} \sum_{n\mid \frac{p-1}{m}} \frac{\mu(n)}{n^r}
\\ &= \sum_{\substack{p\leq x \\ p\equiv1\Mod{m}}} \frac1{(p-1)^r} \,J_r\left(\frac{p-1}{m}\right) \;,
\end{split}
\end{equation}
we have the following Lemma.
\begin{lemma} \label{MainTerm}
If $m\le(\log x)^D$, with $D$ arbitrary positive constant, then for every arbitrary constant $M>1$
$$
S_m(x)= C_{r,m} \Li(x) + O\left( \frac{x}{m^r(\log x)^M}\right) \;,
$$
where $C_{r,m}=\sum_{n\geq 1}\frac{\mu(n)}{(nm)^r\varphi(nm)}$.
\end{lemma}

\begin{proof}
We choose an arbitrary positive constant $B$, and for every coprime integers $a$ and $b$,
we denote $\pi(x;a,b)=\# \{p\leq x : p\equiv a \Mod{b} \}$, then
\begin{eqnarray*}
S_m(x) &=& \sum_{n\leq x} \frac{\mu(n)}{(nm)^r}\, \pi(x;1,nm) \\&=& 
\sum_{n\leq (\log x)^B }\frac{\mu(n)}{(nm)^r}\, \pi(x;1,nm) \\ &&+ O \left( \sum_{(\log x)^B<n\leq x} \frac{1}{(nm)^r}\, \pi(x;1,nm) \right) \;.
\end{eqnarray*}
The sum in the error term is
\begin{eqnarray*}
\sum_{(\log x)^B<n\leq x} \frac{1}{(nm)^r}\, \pi(x;1,nm)& \leq&
 \frac{1}{m^r}\sum_{n>(\log x)^B} \frac{1}{n^r}\sum_{\substack{2\leq a \leq x \\ a \equiv 1 \Mod{mn}}} 1
\\&\leq& \frac{1}{m^{r+1}} \sum_{n>(\log x)^B} \frac{x}{n^{r+1}}\\ &\ll&\frac{x}{m^{r+1}(\log x)^{rB}} \;.
\end{eqnarray*}
For the main term we apply the Siegel--Walfisz Theorem \cite{Walf}, which states that for every arbitrary positive constants $B$ and $C$, if
$a\leq (\log x)^B$, then
$$
\pi(x;1,a) = \frac{\Li(x)}{\varphi(a)} + O\left( \frac{x}{(\log x)^C} \right)\;.
$$
So, if we restrict $m\le (\log x)^D$ for any positive constant $D$,
\begin{eqnarray*}
S_m(x) &=&\sum_{\substack{n\leq(\log x)^B}} \frac{\mu(n)}{(nm)^r\varphi(mn)} \Li(x) +
O\left( \frac{x}{(\log x)^C} \sum_{n\leq (\log x)^B} \frac{1}{(nm)^r}  \right)
\\&&+ O\left( \frac{x}{m^{r+1}(\log x)^{rB}} \right)
\\ &=& 
\,C_{r,m}\Li(x) +
O\left( \sum_{n> (\log x)^B} \frac{\Li(x)}{(nm)^r\varphi(nm)}   \right)
+O\left(\frac{x\log\log x}{m^r(\log x)^{C}} \right)
\\ &&+O\left( \frac{x}{m^{r+1}(\log x)^{rB}} \right)
\\ &=& 
\,C_{r,m} \Li(x) +
O\left(\frac{1}{m^r\varphi(m)} \sum_{n> (\log x)} \frac{\Li(x)}{n^r\varphi(n)}   \right)
+O\left(\frac{x\log\log x}{m^r(\log x)^{C}} \right)
\\&&+O\left( \frac{x}{m^{r+1}(\log x)^{rB}} \right) \;,
\end{eqnarray*}
where we have used the elementary inequality $\varphi(mn)\ge\varphi(m)\varphi(n)$.
Since, for every $n\geq3$, we have (see \cite[Theorem~8.8.7]{BS})
\begin{equation}\label{loglogn}
\frac{n}{\varphi(n)} < e^{\gamma}\log\log n +\frac3{\log\log n}\ll \log\log n\;,
\end{equation}
then
$$
\sum_{n> (\log x)^B} \frac{1}{n^r\varphi(n)} \ll \sum_{n> (\log x)^B} \frac{\log\log n}{n^{r+1}}
\ll \frac{ \log \log \log x}{(\log x)^{rB}}\;.
$$
Thus
$$
\frac{1}{m^r\varphi(m)} \sum_{n>(\log x)^B} \frac{1}{n^r\varphi(n)} \Li(x) 
\ll \frac{x}{m^r\varphi(m)(\log x)^{rB}} \;,
$$
proving the lemma for a suitable choice of $D$, $B$ and $C$.
\end{proof}

The following Lemma concerns the Titchmarsh Divisor Problem \cite{Tit} in the case of primes $p\equiv1\Mod m$. 
Asymptotic results on this topic can be found in \cite{Felix} and \cite{Fiorilli}.

\begin{lemma} \label{tau_sum}
Let $\tau$ be the divisor function and $m\in\N$. If $m\leq (\log x)^D$ for an arbitrary positive constant $D$, 
we have the following inequality:
$$
\sum_{\substack{p\le x\\ p\equiv1\Mod m}}\tau\left(\frac{p-1}m\right)
\le\frac{8x}{m}.
$$
 \end{lemma}

 \begin{proof}
%
Let us write $p-1=mjk$ so that $jk\le(x-1)/m$ and let us set $Q=\sqrt{\frac{x-1}m}$ and distinguish the three cases
\begin{itemize}
\item $j\le Q$, $k> Q$,
\item $j> Q$, $k\le Q$,
\item $j\le Q$, $k\le Q$.
\end{itemize}
So we have the identity
\begin{eqnarray*}
\sum_{\substack{p\le x\\ p\equiv1\Mod m}}\tau\left(\frac{p-1}m\right)&=&
\sum_{\substack{j\le Q}}
\sum_{\substack{Q< k\le\frac{Q^2}j\\ mjk+1\text{ prime}}}1+
\sum_{\substack{k\le Q}}
\sum_{\substack{Q< j\le\frac{Q^2}k\\ mjk+1\text{ prime}}}1\\&&+
\sum_{\substack{j\le Q}}
\sum_{\substack{k\le Q\\ mjk+1\text{ prime}}}1\\
&=& 2\sum_{\substack{k\le Q}}
\sum_{\substack{mkQ+1< p\le x\\ p\equiv 1\Mod{km}}}1+
\sum_{\substack{k\le Q}}
\sum_{\substack{p\le mkQ+1\\ p\equiv 1\Mod{km}}}1\\
&=& 2\sum_{\substack{k\le Q}}\left(\pi(x; 1,km)-\pi(mkQ+1; 1,km)\right)\\&&+
\sum_{\substack{k\le Q}}\pi(mkQ+1; 1,km)\\
&=&2\sum_{\substack{k\le Q}}\pi(x; 1,km)-\sum_{\substack{k\le Q}}\pi(mkQ+1; 1,km)\;.
\end{eqnarray*}
Using the Montgomery--Vaughan version of the Brun--Titchmarsh Theorem:
$$ \pi(x;a,q) \le \frac{2x}{\varphi(q)\log(x/q)},$$
for $m\leq (\log x)^D$ with $D$ arbitrary positive constant, then we obtain
\begin{eqnarray*}
\sum_{\substack{p\le x\\ p\equiv1\Mod m}}\tau\left(\frac{p-1}m\right)&\le& 2\sum_{\substack{k\le Q}}
\frac{2x}{\varphi( km)\log(x/ km)}\\&\le&
\frac{4x}{\log(x/mQ)}\sum_{\substack{k\le Q}}\frac1{\varphi(km)}
\\&\le&\frac{8x}{\log(x/m)}\sum_{\substack{k\le Q}}\frac1{\varphi(km)}\;.
\end{eqnarray*}
Now, substitute the elementary inequality $\varphi(km)\ge m\varphi(k)$ and use a result of Montgomery \cite{Montgomery}
$$
\sum_{k\le Q}\frac1{\varphi(k)}=A\log Q+B+O\left(\frac{\log Q}{Q}\right)\;,
$$
where
$$A=\frac{\zeta(2)\zeta(3)}{\zeta(6)}=1.94360\cdots\quad\text{and}\quad
B=A\gamma-\sum_{n=1}^\infty\frac{\mu^2(n)\log n}{n\varphi(n)}=-0.06056\dots\;,$$
which in particular implies that, for $Q$ large enough,
$$
A\log Q-1\le\sum_{k\le Q}\frac1{\varphi(k)}\le A\log Q\le \log(x/m)\;.
$$
Finally
$$
\sum_{\substack{p\le x\\ p\equiv1\Mod m}}\tau\left(\frac{p-1}m\right)
\le
\frac{8x}{m}\;.
$$
\end{proof}

\begin{lemma} \label{const}

Let $p$ be an odd prime number and let
$$
d_m(\chi) = \sum_{\substack{\chiv \in \left(\widehat{\F_p^*} \right)^r\\ \chi_1=\chi \neq \chi_0}}|c_m(\chiv)| \;,
$$
then
$$
d_m(\chi) \leq \frac1{m} \prod_{\ell\mid \frac{p-1}{m}} \left(1+\frac1{\ell}\right)\;.
$$
\end{lemma}

\begin{proof}

From equation (\ref{cmchi}) and Lemma \ref{MRS}, we have
\begin{align}
\nonumber
d_m(\chi) &= \frac1{(p-1)^r} \sum_{\substack{\nv \in \left(\frac{\Z}{(p-1)\Z} \right)^r\\ n_1 \neq 0}}
\mu^2\left(\frac{(p-1)/m}{\left(\frac{p-1}{m}, \nv\right)} \right)
\frac{J_r\left(\frac{p-1}{m}\right)}{J_r\left(\frac{(p-1)/m}{\left(\frac{p-1}{m}, \nv \right)} \right)}\;;
\end{align}
naming $t=\frac{p-1}{m}$ and $u=\gcd\left(t,n_1\right)$ 
we get
\begin{align}
\nonumber
d_m(\chi) 
=\frac1{(p-1)^r} \sum_{d\mid t} \mu^2\left( \frac{t}{d} \right) \frac{J_r(t)}{J_r\left(\frac{t}{d}\right)} \,H(d)\;,
\end{align}
where
\begin{align}
\nonumber
H(d):=\#\left\{\xv\in\left(\frac{\Z}{(p-1)\Z}\right)^{r-1} :  \left(u,\xv \right)=d \right\} =
\left( \frac{p-1}{d}\right)^{r-1}\sum_{k\mid\frac{u}{d}}\frac{\mu(k)}{k^{r-1}} \;.
\end{align}
Then
\begin{align}
\nonumber
d_m(\chi) &= \frac1{(p-1)} \sum_{d\mid t} \mu^2\left( \frac{t}{d} \right) \frac{J_r(t)}{d^{r-1}J_r\left(\frac{t}{d}\right)}
\sum_{k\mid\frac{u}{d}}\frac{\mu(k)}{k^{r-1}}
\\ \nonumber
&\leq \frac1{p-1} \sum_{d\mid t} \mu^2\left( \frac{t}{d} \right) d= \frac{t}{p-1} \sum_{k\mid t} \frac{\mu^2\left(k \right)}{k}
= \frac1{m} \prod_{\ell \mid \frac{p-1}{m}} \left(1+\frac1{\ell} \right)
\end{align}

\end{proof}


\section{Proof of Theorem $1$}

We follow the method of Stephens \cite{St}. By exchanging the order of summation we obtain that
$$
\sum_{\substack{\av\in\Z^r \\ 0<a_1\le T_1\\ \vdots \\ 0<a_r\le T_r}} N_{\langle \av \rangle,m}(x) =
\sum_{\substack{p\leq x \\ p \equiv 1 \Mod{m}}} M_{p}^m(\Tv) \;,
$$
where $M_{p}^m(\Tv)$ is the number of $r$-tuples $\av \in \Z^r$, with $0<a_i \leq T_i$ and $v_p(a_i)=0$ for each $i=1,\dots,r$,
whose reductions modulo $p$ satisfies $[\F_p^*:\langle \av \rangle_p]=m$. We can write
$$
M_{p}^m(\Tv) = \sum_{\substack{\av\in\Z^r\\ 0<a_1\le T_1\\ \vdots \\ 0<a_r\le T_r}} t_{p,m} (\av)\;,
$$
with
$$
t_{p,m} (\av)=
\begin{cases}
1                             &\text{if $[\F_p^*:\langle \av \rangle_p]=m$\,,}\\
0                             &\text{otherwise\,.}
               \end{cases}
$$
Given a $r$-tuple $\chiv$ of Dirichlet characters mod $p$, by orthogonality relations it is easy to verify that
\begin{equation}
\label{tpm}
t_{p,m}(\av) = \sum_{\chiv \in (\widehat{\F_p^*})^r} c_m(\chiv) \chiv(\av) \;;
\end{equation}
so we have
\begin{equation}\label{mainsum}
\sum_{\substack{\av\in\Z^r \\ 0<a_1\le T_1\\ \vdots \\ 0<a_r\le T_r}} N_{\langle \av \rangle,m}(x) =
\sum_{\substack{p\leq x \\ p \equiv 1 \Mod{m}}}\sum_{\substack{\av\in\Z^r \\ 0<a_1\le T_1\\ \vdots \\ 0<a_r\le T_r}}
\sum_{\chiv \in (\widehat{\F_p^*})^r} c_m(\chiv) \chiv(\av) \;.
\end{equation}
Let $\chiv_0:=(\chi_0,\dots,\chi_0)$ be the $r$-tuple consisting of all principal characters, then
\begin{eqnarray*}
c_m(\chiv_0)&=&\frac{1}{(p-1)^r}\sum_{\substack{\av\in(\F_p^*)^r\\ [\F_p^*:\langle\av\rangle_p]=m}}\chiv_0(\av)\\&=&
\frac1{(p-1)^r}\#\{\av\in (\Z/(p-1)\Z)^r : (\av,p-1)=m \}\\&=& \frac{1}{(p-1)^r} R_p(m)\;.
\end{eqnarray*}
Denoting $|\Tv|:=\prod_{i=1}^r T_i$ and $T^*:=\min\{T_i:i=1,\dots,r\}$, through (\ref{Smx}) we can write the main term in (\ref{mainsum}) as
\begin{eqnarray*}
&&\frac{1}{|\Tv|}\sum_{\substack{p\leq x \\ p \equiv 1 \Mod{m}}}\sum_{\substack{\av\in\Z^r \\ 0<a_1\le T_1\\ \vdots \\ 0<a_r\le T_r}}
c_m(\chiv_0)\chiv_0(\av)\\&=&
\frac{1}{|\Tv|}\sum_{\substack{p\leq x \\ p \equiv 1 \Mod{m}}}c_m(\chiv_0)
\prod_{i=1}^r\left\{\lfloor{T_i}\rfloor- \lfloor{T_i/p}\rfloor\right\}
\\ &=& \sum_{\substack{p\leq x \\ p \equiv 1 \Mod{m}}} c_m(\chiv_0)\left(1-\frac{r}{p}+\dots+\frac1{p^r}+
\sum_{i=1}^r O\left( \frac1{T_i}\right)\right)
\\&=& \sum_{\substack{p\leq x \\ p \equiv 1 \Mod{m}}} c_m(\chiv_0)+ O\left(\sum_{\substack{p\leq x \\ p \equiv 1 \Mod{m}}} \frac{1}{p}\right)
+O\left(\frac{ x}{T^*\log x}\right)\\&=& S_m(x)+O(\log\log x) +O\left(\frac{ x}{T^*\log x}\right)\;.
\end{eqnarray*}
Since by hypothesis $m\le(\log x)^D$, $D>0$, and $T^*>\exp(4(\log x \log\log x)^{1/2})$, 
we can apply Lemma \ref{MainTerm} to obtain
$$
\frac{1}{|\Tv|}\sum_{\substack{p\leq x \\ p \equiv 1 \Mod{m}}}\sum_{\substack{\av\in\Z^r \\ 0<a_1\le T_1\\ \vdots \\ 0<a_r\le T_r}}
c_m(\chiv_0)\chiv_0(\av)= C_{r,m} \Li(x) + O\left( \frac{x}{m^r(\log x)^M}\right) \;,
$$
where $M>1$.
For the error term we need to estimate the sum
\begin{eqnarray*}
E_{r,m}(x) &:= &\frac{1}{|\Tv|} \sum_{\substack{p\leq x \\ p \equiv 1 \Mod{m}}} \sum_{\chiv \in \left(\widehat{\F_p^*} \right)^r \setminus \{\chiv_0\} }
 \left| c_m(\chiv) \sum_{\substack{\av \in \Z^r \\ 0<a_1\le T_1\\ \vdots \\ 0<a_r\le T_r}}\chiv(\av) \right|
\\&&\ll \sum_{i=1}^r\frac{1}{T_i}\sum_{\substack{p\leq x \\ p \equiv 1 \Mod{m}}} \sum_{\chi_i \in \widehat{\F_p^*}\setminus\{\chi_0\}}
d_m(\chi_i) \left| \sum_{\substack{a \in \Z \\ 0<a \leq T_i}}\chi_i(a) \right| \;,
\end{eqnarray*}
where
$$
d_m(\chi)=\sum_{\substack{\chiv \in \left(\widehat{\F_p^*} \right)^r\\ \chi_1=\chi \neq \chi_0}}|c_m(\chiv)|
\;.
$$
Define
\begin{equation}\label{S2i}
E^j_{r,m}(x):=\sum_{\substack{p\leq x \\ p \equiv 1 \Mod{m}}} \sum_{\chi_i \in \widehat{\F_p^*}\setminus\{\chi_0\}} d_m(\chi_i)
 \left| \sum_{\substack{a \in \Z \\ 0<a \leq T_i}}\chi_i(a) \right| \;,
\end{equation}
then by Holder's inequality
\begin{equation} \label{Holder}
\begin{split}
\left\{E^j_{r,m}(x)\right\}^{2s_i}\leq& \left \{ \sum_{\substack{p\leq x \\ p \equiv 1 \Mod{m}}}
\sum_{\chi_i \in \widehat{\F_p^*}\setminus\{\chi_0\}} \{d_m(\chi_i)\}^{\frac{2s_i}{2s_i-1}} \right\}^{2s_i-1}\\&\times
\sum_{\substack{p\leq x \\ p \equiv 1 \Mod{m}}} \sum_{\chi_i \in \widehat{\F_p^*}\setminus\{\chi_0\}}
\left| \sum_{\substack{a \in \Z \\ 0<a\leq T_i}} \chi_i(a) \right|^{2s_i}\;.
\end{split}
\end{equation}
As before, given a primitive root $g$ modulo $p$, write  $\chi_j(g)=e^{2\pi i  n_j/(p-1)}$ for every $j=1,\dots,r$, 
with $n_j \in \Z/(p-1)\Z$, so that by equation (\ref{cmchi}) 
$$
\sum_{\chiv \in \left(\widehat{\F_p^*} \right)^r \setminus \{\chiv_0\} } c_m(\chiv) =
\frac1{(p-1)^r}\sum_{\nv \in \left(\frac{\Z}{(p-1)\Z}\right)^r \setminus\{\zv\}} c_{\frac{p-1}{m}}(\nv)\;.
$$
Denoting again $t=(p-1)/m$, from Lemma \ref{MRS} derives the following upper bound:
\begin{eqnarray*}
\sum_{\chi_i \in \widehat{\F_p^*}\setminus\{\chi_0\}} d_m(\chi_i)&
\leq&\sum_{\chiv \in \left(\widehat{\F_p^*} \right)^r \setminus \{\chiv_0\} } |c_m(\chiv)|
\\ &
\leq &\sum_{d\mid t} \mu^2 \left(\frac{t}{d}\right) \left[\frac{J_r(t)}{(p-1)^rJ_r(t/d)} \right]\\ &&\times
\#\left\{\nv\in \left(\Z/(p-1)\Z\right)^r: (t,\nv)=d\right\}
\\ &=&\sum_{d\mid t} \mu^2 \left(\frac{t}{d}\right) \frac{J_r(t)}{d^rJ_r(t/d)}  \sum_{k\mid\frac{t}{d}} \frac{\mu(k)}{k^r}
= \frac{J_r(t)}{t^r}\sum_{d\mid t} \mu^2 \left(\frac{t}{d}\right)
\\  &
=&\prod_{\ell\mid t}\left(1-\frac1{\ell^r}\right) 2^{\omega(t)}
\leq 2^{\omega(t)}\;.
\end{eqnarray*}

Calling $D_m(p):=\max\{d_m(\chi):\chi \in \widehat{\F_p^*}\setminus\{\chi_0\}\}$ and using Lemmas \ref{const} and \ref{tau_sum},
the following asymptotic estimate holds for every $s_i\geq1$:
\begin{eqnarray*}
&&\sum_{\substack{p\leq x \\ p \equiv 1 \Mod{m}}} \sum_{\chi \in \widehat{\F_p^*}\setminus\{\chi_0\}} \{d_m(\chi)\}^{\frac{2s_i}{2s_i-1}}
\\ &&\leq\sum_{\substack{p\leq x \\ p \equiv 1 \Mod{m}}} \sum_{\chi \in \widehat{\F_p^*}\setminus\{\chi_0\}}
 d_m(\chi)\{d_m(\chi)\}^{\frac{1}{2s_i-1}}
\\ &&\leq \sum_{\substack{p\leq x\\ p\equiv 1\Mod{m}}}\{D_m(p)\}^{\frac{1}{2s_i-1}} \sum_{\chi \in \widehat{\F_p^*}\setminus\{\chi_0\}}d_m(\chi)
\\&&\leq \sum_{\substack{p\leq x \\ p \equiv 1\Mod{m}}}\{D_m(p)\}^{\frac{1}{2s_i-1}}2^{\omega(\frac{p-1}{m})}
\\ && \le\frac1{m} \sum_{\substack{p\leq x \\ p \equiv 1\Mod{m}}}
\prod_{\ell\mid\frac{p-1}{m}}\left(1+\frac1{\ell}\right)\, 2^{\omega(\frac{p-1}{m})}
\\&&\ll \frac1{m} \sum_{\substack{p\leq x \\ p \equiv 1\Mod{m}}}
\prod_{\ell\mid\frac{p-1}{m}}\left(1-\frac1{\ell}\right)^{-1} 2^{\omega(\frac{p-1}{m})}
\\ && \ll \frac{\log\log x}{m}   \sum_{\substack{p\leq x \\ p \equiv 1\Mod{m}}}\tau\left(\frac{p-1}{m}\right)
\ll \frac{x \log \log x}{m^{2}}\;.
\end{eqnarray*}
To estimate the other term in (\ref{Holder}) we use Lemma 5 in \cite{St}:
$$
\sum_{\substack{p\leq x \\ p \equiv 1 \Mod{m}}} \sum_{\chi_i\in \widehat{\F_p^*}\setminus\{\chi_0\}}
\left| \sum_{\substack{a \in \Z \\ 0<a\leq T_i}} \chi_i(a) \right|^{2s_i}
\ll (x^2+{T_i}^{s_i}){T_i}^{s_i} (\log(e{T_i}^{s_i-1}))^{{s_i}^2-1}\;.
$$
So, for every positive constant $M>1$, we find
\begin{eqnarray*}
\frac{1}{|\Tv|} \sum_{\substack{\av\in\Z^r \\ 0<a_1\le T_1\\ \vdots \\ 0<a_r\le T_r}} N_{\langle \av \rangle,m}(x) &=&C_{r,m} \Li(x) +
O\left(\frac{x}{m^r(\log x)^M}\right)\\&&+ O\left(\sum_{i=1}^r \frac{x}{T_i\log x}\right) + E_{r,m}(x)\;,
\end{eqnarray*}
with
$$
E_{r,m}(x)\ll \sum_{i=1}^r \frac{1}{T_i}\left[\left(\frac{x\log\log x}{m^{2}}\right)^{2s_i-1} (x^2+{T_i}^{s_i}){T_i}^{s_i}
(\log(e{T_i}^{s_i-1}))^{{s_i}^2-1}\right]^{\frac1{2s_i}}\;.
$$
If we choose $s_i= \left\lfloor\frac{2\log x}{\log T_i}\right\rfloor+1$ for $i=1,\dots,r$, then $T_i^{s_i-1}\le x^2 < T_i^{s_i}$ and
$$
E_{r,m}(x) \ll  \frac1{m}\sum_{i=1}^r(x\log\log x)^{1-\frac1{2s_i}} (\log(e{x^2}))^{\frac{{s_i}^2-1}{2s_i}}\;.
$$
Now, if $T_i>x^2$ for all $i=1,\dots,r$, then $s_1=\dots=s_r=1$ and
$$
E_{r,m}(x)\ll \frac1{m} (x\log\log x)^{1/2}\;;
$$
in particular, we have $E_{r,m}(x)\ll x/(\log x)^M$ for every constant $M>1$.
Otherwise, if $T_j\le x^2$ for some $j\in\{1,\dots,r\}$, then $s_j\ge 2$ and the corresponding contribution to $E_{r,m}(x)$ will be
$$
E^j_{r,m}(x)\ll \frac1{m} (x\log\log x)^{1-\frac1{2s_j}} (\log(e{x^2}))^{\frac{3\log x}{2\log T_j}}\;.
$$
By hypothesis
\begin{equation} \label{Tmin}
T^* > \exp(4(\log x \log\log x)^{1/2})
\end{equation}
and, through  computations similar to those in \cite{St} (page 184), we can derive the following estimate: 
\begin{align}
\nonumber
E_{r,m}(x) 
\ll \frac1{m} \,x \log\log x (T^*)^{-\frac1{16}} \;.
\end{align}
Also in this case, using (\ref{Tmin}), we have $E_{r,m}(x)\ll x/(\log x)^M$ for every $M>1$.
This ends the proof of Theorem~\ref{firstorder}. 
\;\;\;\;\;\;\;\;$\Box$


\section{Proof of Theorem $2$}
We now consider 
\begin{eqnarray*}
&H:&=\frac{1}{|\Tv|}\sum_{\substack{\av\in\Z^r\\ 0<a_1\le T_1\\ \vdots \\ 0<a_r\le T_r}}
\left\{N_{\langle \av\rangle,m}(x)-C_{r,m}\operatorname{Li}(x)\right\}^2.
\end{eqnarray*}
\begin{eqnarray*}
&&\sum_{\substack{\av\in\Z^r\\ 0<a_1\le T_1\\ \vdots \\ 0<a_r\le T_r}}
\left\{N_{\langle \av\rangle,m}(x)-C_{r,m}\operatorname{Li}(x)\right\}^2 \\
&&\leq \sum_{\substack{p,q\leq x \\ p,q \equiv 1 \Mod{m}}}M^m_{p,q}(\Tv)- 2C_{r,m}\operatorname{Li}(x)\sum_{\substack{p\leq x \\ p \equiv 1 \Mod{m}}} M^m_p(\Tv)+ |\Tv|  (C_{r,m})^2 \operatorname{Li}^2(x) \;,
\end{eqnarray*}
where $M^m_{p,q}(\Tv)$ denotes the number of $r$-tuples $\av \in \Z^r$, with $a_i \leq T_i$ and $v_p(a_i)=v_q(a_i)=0$
for each $i=1,\dots,r$, whose reductions modulo $p$ and $q$ satisfy $[\F_p^*:\langle \av \rangle_p]=[\F_q^*:\langle \av \rangle_q]=m$. 
From Theorem~\ref{firstorder} we obtain 
\begin{align} \nonumber
H&\leq \frac1{|\Tv|}\sum_{\substack{p,q\leq x \\ p,q \equiv 1 \Mod{m}}}M^m_{p,q}(\Tv)-(C_{r,m})^2 \operatorname{Li}^2(x)+ O\left(  \frac{x^2}{(\log x)^{M'}}\right)\;,
\end{align}
for every constant $M'>2$.
If we  write
$$
\sum_{\substack{p,q\leq x \\ p,q \equiv 1 \Mod{m}}}M^m_{p,q}(\Tv)=\sum_{\substack{p\leq x \\ p \equiv 1 \Mod{m}}}M^m_p(\Tv)+\sum_{\substack{p,q\leq x \\ p,q \equiv 1 \Mod{m} \\ p\neq q}}M^m_{p,q}(\Tv)\;,
$$
Theorem~\ref{firstorder} gives, for arbitrary $M>1$,
$$
\sum_{\substack{p\leq x \\ p \equiv 1 \Mod{m}}}M^m_p(\Tv) =
C_{r,m} |\Tv| \operatorname{Li}(x)+O\left(\frac{|\Tv|x}{(\log x)^M}\right)\;.
$$
In the same spirit as in the proof Theorem~\ref{firstorder}, we use equation (\ref{tpm}) to deal with the following sum
\begin{eqnarray*}
&&\sum_{\substack{p,q\leq x \\ p,q \equiv 1 \Mod{m} \\ p\neq q}}M^m_{p,q}(\Tv)\\&&=\sum_{\substack{p,q\leq x \\ p,q \equiv 1 \Mod{m} \\ p\neq q}}
\sum_{\substack{\av\in\Z^r \\ 0<a_1\le T_1\\ \vdots \\ 0<a_r\le T_r}}t_{p,m}(\av) t_{q,m}(\av)
\\ \label{doublesum}
&&=\sum_{\substack{p,q\leq x \\ p,q \equiv 1 \Mod{m} \\ p\neq q}}
\sum_{\chiv_1 \in (\widehat{\F_p})^r} \sum_{\chiv_2 \in (\widehat{\F_q})^r} c_m(\chiv_1)c_m(\chiv_2)
\sum_{\substack{\av\in\Z^r \\ 0<a_1\le T_1\\ \vdots \\ 0<a_r\le T_r}}\chiv_1(\av) \chiv_2(\av)\;.
\end{eqnarray*}\\
Therefore
$$\sum_{\substack{p,q\leq x \\ p,q \equiv 1 \Mod{m}}}M^m_{p,q}(\Tv)= H_1+2H_2+H_3+O(|\Tv| \operatorname{Li}(x))\;,
$$
where $H_1, H_2, H_3$ are the contributions to the sum (\ref{doublesum}) 
when $\chiv_1=\chiv_2=\chiv_0$, only one between $\chiv_1$ and $\chiv_2$
is equal to $\chiv_0$, neither $\chiv_1$ nor $\chiv_2$ is $\chiv_0$, respectively.
First we deal with the inner sum in $H_1$. To avoid confusion, we set $\chiv^p_0$ and $\chiv^q_0$ as the $r$-tuples whose all entries are
principal characters modulo ${p}$ and modulo ${q}$ respectively, so that
 \begin{align}
 \nonumber
\sum_{\substack{\av\in\Z^r \\ 0<a_1\le T_1\\ \vdots \\ 0<a_r\le T_r}}\chiv^p_0(\av) \chiv^q_0(\av)=&
\prod_{i=1}^r\left\{\lfloor T_i \rfloor - \left\lfloor\frac{T_i}{p}\right\rfloor -\left\lfloor\frac{T_i}{q}\right\rfloor
+ \left\lfloor\frac{T_i}{pq}\right\rfloor \right\}
\;.
\end{align}
Using Lemma \ref{MainTerm}, with $M'>2$ arbitrary constant:
\begin{eqnarray*}
&&H_1=\sum_{\substack{p,q\leq x \\ p,q \equiv 1 \Mod{m} \\ p\neq q}} c_m(\chiv^p_0)c_m(\chiv^q_0) \sum_{\substack{\av\in\Z^r \\ 0<a_1\le T_1\\ \vdots \\ 0<a_r\le T_r}}\chiv^p_0(\av) \chiv^q_0(\av)
\\&=&|\Tv|\sum_{\substack{p,q\leq x \\ p,q \equiv 1 \Mod{m} \\ p\neq q}} c_m(\chiv^p_0)c_m(\chiv^q_0)
\left(1-\frac{r}{p}-\frac{r}{q}+ \dots+\frac1{(pq)^r}+\sum_{i=1}^r O\left( \frac1{T_i}\right)\right)
\\ \nonumber
&=&|\Tv|\left( \left(\sum_{\substack{p\leq x \\ p \equiv 1 \Mod{m}}} c_m(\chiv_0)\right)^2- \sum_{\substack{p\leq x \\ p \equiv 1 \Mod{m}}}(c_m(\chiv^p_0))^2\right)\\ &&\times \left(1+O\left(\frac1{T^*}\right)\right)+|\Tv|O\left( \frac{x \log\log x}{\log x}\right)
\\ &=&|\Tv|\left( S_m^2(x)+ O\left(\frac{x^2}{T^*(\log x)^2}\right)+O\left( \frac{x \log\log x}{\log x}\right)\right)
\\&=&|\Tv|\left(C^2_{r,m} \Li^2(x)+O\left( \frac{x^2}{m^r(\log x)^{M'}} \right)\right)\;.
\end{eqnarray*}
Focuse now on $H_2$ and assume without loss of generality that $\chiv_1=\chiv_0\neq\chiv_2$:
\begin{align} \nonumber
H_2&= \sum_{\substack{p,q\leq x \\ p,q \equiv 1 \Mod{m} \\ p\neq q}} \sum_{\substack{\chiv_2 \in (\widehat{\F_q^*})^r\setminus\{\chiv_0^q\}}}
c_m(\chiv^p_0)c_m(\chiv_2) \sum_{\substack{\av\in\Z^r \\ 0<a_1\le T_1\\ \vdots \\ 0<a_r\le T_r}} \chiv^p_0(\av)\chiv_2(\av)
\\ \nonumber
&= \sum_{\substack{p\leq x \\ p \equiv 1 \Mod{m}}}c_m(\chiv^p_0)\sum_{\substack{q\leq x \\ q \equiv 1 \Mod{m} \\ q\neq p}} \sum_{\substack{\chiv_2 \in (\widehat{\F_q^*})^r\setminus\{\chiv_0^q\}}}
c_m(\chiv_2) \sum_{\substack{\av\in\Z^r \\ 0<a_1\le T_1\\ \vdots \\ 0<a_r\le T_r\\ p\nmid \prod_{i=1}^r a_i}} \chiv_2(\av)\;.
\end{align}
Identically to what was done in the proof of Theorem 1, the quantity
$$
U_2:=\sum_{\substack{q\leq x \\ q \equiv 1 \Mod{m} }} \sum_{\substack{\chiv_2 \in (\widehat{\F_q^*})^r\setminus\{\chiv_0^q\}}}
\left|c_m(\chiv_2) \sum_{\substack{\av\in\Z^r \\ 0<a_1\le T_1\\ \vdots \\ 0<a_r\le T_r}} \chiv_2(\av)\right|
$$
can be estimated through Holder's inequality combined with the large sieve inequality, to get $U_2 \ll x/(\log x)^M$ for any constant $M>1$. 
Moreover, Lemma \ref{tau_sum} gives an upper bound for the following quantity:
\begin{eqnarray*}
V_2&:=&\sum_{\substack{q\leq x \\ q \equiv 1 \Mod{m} }} \sum_{\substack{\chiv_2 \in (\widehat{\F_q^*})^r\setminus\{\chiv_0^q\}}}
\left|c_m(\chiv_2) \sum_{\substack{\av\in\Z^r \\ 0<a_1\le T_1\\ \vdots \\ 0<a_r\le T_r \\ p\mid\prod_{i=1}^r a_i}} \chiv_2(\av)\right|
\\ &\ll& \frac{|\Tv|}{p^r}\sum_{\substack{q\leq x\\ q \equiv 1 \Mod{m} }}
\sum_{\substack{\chiv_2 \in(\widehat{\F_q^*})^r\setminus\{\chiv_0^q\}}}|c_m(\chiv_2)|
\\ &\ll& \frac{|\Tv|}{p^r} \sum_{\substack{q\leq x \\ q \equiv 1 \Mod{m} }}\tau\left(\frac{q-1}{m}\right)
\ll \frac{|\Tv|x}{p^r m}\;.
\end{eqnarray*}
Thus, for every constant $M'>2$,
$$
H_{2} \leq \sum_{\substack{p\leq x \\ p \equiv 1 \Mod{m}}} (U_2+V_2)
\ll\frac{|\Tv|x^2}{ (\log x)^{M'}}\;.
$$
Finally, assume $\chi_1 \in \widehat{\F_p^*}\setminus\{\chi_0^p\}$ and $\chi_2 \in \widehat{\F_q^*}\setminus\{\chi_0^q\}$,
with $p\neq q$, then $\chi_1\chi_2$ is a primitive character modulo $pq$.
Given
$$
H_3= \sum_{\substack{p,q\leq x \\ p,q \equiv 1 \Mod{m} \\ p\neq q}} \sum_{\substack{\chiv_1 \in (\widehat{\F_p^*})^r\setminus\{\chiv_0^p\}}}
\sum_{\substack{\chiv_2 \in (\widehat{\F_q^*})^r\setminus\{\chiv_0^q\}}}
c_m(\chiv_1)c_m(\chiv_2) \sum_{\substack{\av\in\Z^r \\ 0<a_1\le T_1\\ \vdots \\ 0<a_r\le T_r}} \chiv_1(\av)\chiv_2(\av)
$$
we will apply again Holder's inequality and the large sieve (Lemma 5 in \cite{St}) to obtain an upper bound.
In order to do that, since the $r$-tuples of characters, $\chiv_1$ and $\chiv_2$, appearing in $H_3$ are both non-principal,
we indicate with $\chi_{1,i}$ the $i$-th component of the $r$-tuple $\chiv_1$ of Dirichlet characters to the modulus $p$
(similarly for $\chi_{2,i}$). Then the contributions to $H_3$ have two possible sources: a ``diagonal'' term $H_3^d$
(in which for a certain $i\in\{1,\dots,r\}$ both $\chi_{1,i}$ and $\chi_{2,i}$ are non-principal)
and a ``non-diagonal'' term $H_3^{nd}$ (in which for none of the indices $i\in\{1,\dots,r\}$ is possible to have
$\chi_{1,i}$ and $\chi_{2,i}$ both non-principal). 
Explicitly, $H_3^d = \sum_{i=1}^r H_{3,i}$, where
\begin{align} \nonumber
H_{3,i}&:=
\sum_{\substack{p,q\leq x \\ p,q \equiv 1 \Mod{m} \\ p\neq q}}
\sum_{\substack{\chiv_1 \in (\widehat{\F_p^*})^r\\ \chi_{1,i}\in \widehat{\F_p^*}\setminus\{\chi_0^p\}}}
\sum_{\substack{\chiv_2 \in (\widehat{\F_q^*})^r \\ \\ \chi_{2,i}\in \widehat{\F_q^*}\setminus\{\chi_0^q\}}}
c_m(\chiv_1)c_m(\chiv_2) \sum_{\substack{\av\in\Z^r \\ 0<a_1\le T_1\\ \vdots \\ 0<a_r\le T_r}} \chiv_1(\av)\chiv_2(\av)
\\ \nonumber
&\leq \frac{|\Tv|}{T_i}
\sum_{\substack{p,q\leq x \\ p,q \equiv 1 \Mod{m} \\ p\neq q}}\sum_{\chi_{1,i}\in \widehat{\F_p^*}\setminus\{\chi_0^p\}}
\sum_{\chi_{2,i}\in \widehat{\F_q^*}\setminus\{\chi_0^q\}}d_m(\chi_{1,i})d_m(\chi_{2,i})
\\ \nonumber & \qquad \times \left| \sum_{0<a_i\leq T_i}\chi_{1,i}(a_i)\chi_{2,i}(a_i)\right| 
\end{align}

and $H_3^{nd}=\sum_{\substack{i,j=1 \\ i\neq j}}^rH_{3,ij}$, with
\begin{align}
\nonumber
H_{3,ij}:= & \sum_{\substack{p,q\leq x \\ p,q \equiv 1 \Mod{m} \\ p\neq q}}
\sum_{\substack{\chiv_1 \in (\widehat{\F_p^*})^r \\ \chi_{1,i}\in \widehat{\F_p^*}\setminus\{\chi_0^p\}}}
\sum_{\substack{\chiv_2 \in (\widehat{\F_q^*})^r \\ \\ \chi_{2,j}\in \widehat{\F_q^*}\setminus\{\chi_0^q\}}}
c_m(\chiv_1)c_m(\chiv_2) \sum_{\substack{\av\in\Z^r \\ 0<a_1\le T_1\\ \vdots \\ 0<a_r\le T_r}} \chiv_1(\av)\chiv_2(\av)
\\ \nonumber
& \leq  \frac{|\Tv|}{T_i T_j}
\sum_{\substack{p,q\leq x \\ p,q \equiv 1 \Mod{m} \\ p\neq q}}
\sum_{\chi_{1,i}\in \widehat{\F_p^*}\setminus\{\chi_0^p\}}\sum_{\chi_{2,j}\in \widehat{\F_q^*}\setminus\{\chi_0^q\}}
d_m(\chi_{1,i})d_m(\chi_{2,j}) 
\\ \nonumber 
& \qquad \times 
\left| \sum_{\substack{0<a_i\leq T_i \\ 0<a_j\leq T_j}}\chi_{1,i}(a_i)\chi_{2,j}(a_j)\right| \;.
\end{align}
Dealing first with $H_{3,i}$, we use again Holder's inequality together with the large sieve to get
\begin{eqnarray*}
\frac{H_{3,i}}{|\Tv|} &\ll&
\frac1{T_i}\left\{\sum_{\substack{p,q\leq x \\ p,q \equiv 1 \Mod{m} \\ p\neq q}}
\sum_{\substack{\chi_{1,i}\in \widehat{\F_p^*}\setminus\{\chi_0^p\} \\ \chi_{2,i}\in \widehat{\F_q^*}\setminus\{\chi_0^q\}}}
[d_m(\chi_{1,i})d_m(\chi_{2,i})]^{\frac{2s_i}{2s_i-1}}\right\}^{\frac{2s_i-1}{2s_i}}\\ &&\times
\left\{\sum_{\substack{p,q\leq x \\ p,q \equiv 1 \Mod{m} \\ p\neq q}} \sum_{\eta\Mod{pq}}
\left| \sum_{0<a_i\leq T_i}\eta(a_i)\right|^{2s_i}\right\}^{\frac1{2s_i}}\\
\nonumber
&\ll& \frac1{T_i}\left\{ \left(\frac{x\log\log x}{m^{2}}\right)^{4s_i-2}(x^4+T_i^{s_i})T_i^{s_i}
(\log(eT_i^{s_i-1}))^{s_i^2-1}\right\}^{\frac1{2s_i}}\;.
\end{eqnarray*}
We now choose $s_i=\left\lfloor \frac{4\log x}{\log T_i}\right\rfloor+1$, so that $T_i^{s_i-1}\le x^4 \le T_i^{s_i}$ and
$$
\frac{H_{3,i}}{|\Tv|} \ll \frac1{m^2}\,x^{2-\frac1{s_i}}(\log \log x)^2(\log(ex^{4}))^{\frac{s_i^2-1}{2s_i}}\;.
$$
If $T_i>x^4$ then $s_i=1$ and $H_{3,i}/|\Tv|\ll x (\log \log x)^2$.
Otherwise, if $T_i\le x^4$ then $s_i\ge 2$ and assuming by hypothesis $T_i>\exp(6(\log x \log\log x)^{1/2})$, 
similarly to what was done to prove Theorem 1 we get
$$
\frac{H_{3,i}}{|\Tv|} \ll x^{2-\frac1{s_i}}(\log \log x)^2(\log(ex^{4}))^{\frac{3\log x}{\log T_i}} \ll
\frac{x^2}{(\log x)^{D}}\;,
$$
for any positive constant $D>2$.
It remains to estimate $H_{3,ij}$, where $i\neq j$: it can be factorized in two products
and, through the same methods used with (\ref{S2i}), we have
\begin{eqnarray*}
\frac{H_{3,ij}}{|\Tv|}&\ll &
\frac1{T_iT_j} \sum_{\substack{p\leq x \\ p \equiv 1 \Mod{m} }} \sum_{\chi_{1,i}\in \widehat{\F_p^*}\setminus\{\chi_0^p\}}
d_m(\chi_{1,i}) \left| \sum_{0<a_i\leq T_i} \chi_{1,i}(a_i) \right|\\ &&\times
\sum_{\substack{q\leq x \\ q \equiv 1 \Mod{m} }} \sum_{\chi_{2,j}\in \widehat{\F_q^*}\setminus\{\chi_0^q\}}
d_m(\chi_{2,j}) \left| \sum_{0<a_j\leq T_j} \chi_{2,j}(a_j) \right|
\\ &\ll& \frac1{T_i}\left\{ \left(\frac{x\log\log x}{m^{2}}\right)^{2s_i-1}(x^2+T_i^{s_i})T_i^{s_i}
(\log(eT_i^{s_i-1}))^{s_i^2-1}\right\}^{\frac1{2s_i}} \\ && \times
\frac1{T_j}\left\{ \left(\frac{x\log\log x}{m^{2}}\right)^{2s_j-1}(x^2+T_j^{s_j})T_j^{s_j}
(\log(eT_j^{s_j-1}))^{s_j^2-1}\right\}^{\frac1{2s_j}}\;.
\end{eqnarray*}
We choose $s_i=\left\lfloor \frac{2\log x}{\log T_i}\right\rfloor+1$ and
$s_j=\left\lfloor \frac{2\log x}{\log T_j}\right\rfloor+1$, so that
$$
\frac{H_{3,ij}}{|\Tv|}\ll \frac{x^2}{(\log x)^E}
$$
for every constant $E>2$.

Eventually, since $H_3 \leq H_3^{d}+H_3^{nd}$, summing the  upper bounds for $H_1$, $H_2$ and $H_3$
we get the proof of Theorem~\ref{secondorder}. 
\;\;\;\;\;\;\;\;$\Box$

\subsection*{Acknowledgements}
The results in this manuscript are part of the Doctoral dissertation of the two authors.
The authors would like to thank Prof. Francesco Pappalardi for inspiring this work and for the precious suggestions 
about technical difficulties concerning the proofs of the Lemmas and Theorems.

.

\end{document}